\documentclass[12pt]{article}
\usepackage{e-jc}

\usepackage{amsmath,graphicx}

\dateline{May 7, 2021}{Feb 18, 2022}{Mar 25, 2022}


\Copyright{The authors. Released under the CC BY license (International 4.0).}

\title{Further approximations for Aharoni's rainbow generalization of the Caccetta-H\"{a}ggkvist conjecture}

\author{
Patrick Hompe \qquad  Sophie Spirkl\thanks{We acknowledge the support of the Natural Sciences and Engineering Research Council of Canada (NSERC), [funding reference number RGPIN-2020-03912]. Cette recherche a été financée par le Conseil de recherches en sciences naturelles et en génie du Canada (CRSNG), [numéro de référence RGPIN-2020-03912].}\\
\small Department of Combinatorics and Optimization\\[-0.8ex]
\small University of Waterloo\\[-0.8ex]
\small Waterloo, ON, Canada\\
\small\tt \{phompe,sspirkl\}@uwaterloo.ca}

\begin{document}

\maketitle


\begin{abstract}
  For a digraph $G$ and $v \in V(G)$, let $\delta^+(v)$ be the number of out-neighbors of $v$ in $G$. The Caccetta-H\"{a}ggkvist conjecture states that for all $k \ge 1$, if $G$ is a digraph with $n = |V(G)|$ such that $\delta^+(v) \ge k$ for all $v \in V(G)$, then $G$ contains a directed cycle of length at most $\lceil n/k \rceil$. Aharoni proposed a generalization of this conjecture, that a simple edge-colored graph on $n$ vertices with $n$ color classes, each of size at least $k$, has a rainbow cycle of length at most $\lceil n/k \rceil$. With Pelik\'anov\'a and Pokorn\'a, we showed that this conjecture is true if each color class has size ${\Omega}(k\log k)$. In this paper, we present a proof of the conjecture if each color class has size ${\Omega}(k)$, which improved the previous result and is only a constant factor away from Aharoni's conjecture. We also consider what happens when the condition on the number of colors is relaxed.
\end{abstract}

\section{Introduction and preliminaries}
We call a graph \emph{simple} if it has no loops or parallel edges, and we call a digraph \emph{simple} if its underlying undirected graph is simple. For a simple digraph $G$ and a vertex
$v \in V(G)$, let $\delta^+(v)$ denote the number of out-neighbors of $v$ in $G$. A famous conjecture in graph theory is the following, due to Caccetta and H\"{a}ggkvist:

\begin{conjecture}[\cite{ch}]\label{ch_thm}
Let $n,k$ be positive integers, and let $G$ be a simple digraph on $n$ vertices with $\delta^+(v) \ge k$ for all $v \in V(G)$. Then $G$ contains a directed cycle of length at most $\lceil n/k \rceil$.
\end{conjecture}

For a graph $G$ and a function $b: E(G) \rightarrow \mathbb{N}$, a \emph{rainbow cycle (with respect to $b$)} is a cycle $C$ in $G$ such that for all $e, f \in E(C)$ with $e \neq f$, we have $b(e) \neq b(f)$. We will refer to $b$ as a \emph{coloring} of the edges of $G$.\footnote{Note that $b$ is not required to be a proper edge-coloring.} For any $i \in \{1,\cdots,K\}$, the set of edges $b^{-1}(i)$ is called a \emph{color class}. For $k, K \in \mathbb{N}$, we say that $b$ has \emph{$K$ color classes of size at least $k$} if $|b^{-1}(i)| \geq k$ for all $i \in \{1, \dots, K\}$ and $b^{-1}(i) = \emptyset$ for all $i > K$. The \emph{rainbow girth} of an edge-colored graph $G$ is the length of a shortest rainbow cycle in $G$.

In [\ref{aharoni}], Aharoni proposed a generalization of Conjecture \ref{ch_thm}:

\begin{conjecture}[\cite{aharoni}]\label{aharoni_thm}
Let $n,k$ be positive integers, and let $G$ be a simple graph on $n$ vertices. Let $b$ be a coloring of the edges of $G$ with $n$ color classes of size at least $k$. Then $G$ has rainbow girth at most $\lceil n/k \rceil$.
\end{conjecture}

In [\ref{devos}], Conjecture \ref{aharoni_thm} was proved for $k=2$. The following approximate result shows that the conjecture holds with larger color classes:

\begin{theorem}[\cite{hompe}]\label{hompe_thm}
Let $k > 1$ be an integer, and let $G$ be a simple graph on $n$ vertices. Suppose that we have a coloring of the edges of $G$ with $n$ color classes of size at least $301k \log k$. Then $G$ has rainbow girth at most $\lceil n/k \rceil$.
\end{theorem}

In [\ref{chvatal}], Chv\'{a}tal and Szemer\'{e}di show that in a simple digraph of minimum out-degree at least $k$, there exists a directed cycles of length at most $2n/k$. Our main result is an improvement of Theorem \ref{hompe_thm}, reducing the required size of the color classes from ${\Omega}(k\log k)$ to ${\Omega}(k)$.
\begin{theorem}\label{main_result}
Let $k \ge 1$ be an integer, and let $G$ be a simple graph on $n$ vertices. Suppose we have a coloring of the edges of $G$ with $n$ color classes of size at least $ck$, where $c = 10^{11}$. Then $G$ has rainbow girth at most $n/k$.
\end{theorem}

Our paper is organized as follows. In the remainder of Section $1$, we state some results we will use throughout the paper. In Section $2$, we prove a number of results on the case where the number of colors is $n+ck$ for some constant $c$. Then, in Section $3$, we prove the main result of the paper, which deals with the case of $n$ color classes. Finally, in Section $4$, we present some ideas for future work.

The proof, at a high level, proceeds by a reduction from the case of $n$ colors to the case of $n+c'k$ colors for some large constant $c'$, by applying existing results for the Caccetta-H\"aggkvist conjecture using the method of \cite{aharoni}. Then, in the case of $n+c'k$ colors, we use several methods to establish results which are substantially stronger.

We will make use of the following results due to Bollob\'{a}s and Szemer\'{e}di [\ref{bollobas}] and Shen [\ref{shen}], respectively. The first deals with the girth of a simple graph, which is our primary tool in finding short cycles, while the second is an approximate result for Conjecture \ref{ch_thm}. In this paper, $\log$ denotes the logarithm with base $2$.

\begin{theorem}[\cite{bollobas}]\label{bollobas_thm}
For all $n \ge 4$ and $k \ge 2$, if $G$ is a simple graph on $n$ vertices with $n+k$ edges, then $G$ contains a cycle of length at most $$\frac{2(n+k)}{3k}(\log k+\log \log k + 4).$$
\end{theorem}

\begin{theorem}[\cite{shen}]\label{shen_thm}
Let $G$ be a simple digraph with $\delta^+(v) \ge k$ for all $v \in V(G)$. Then $G$ contains a directed cycle of length at most $\lceil n/k \rceil + 73$.
\end{theorem}

We will use the following immediate corollary of Theorem \ref{bollobas_thm}:
\begin{corollary}\label{bollobas_cor}
For all $n \ge 4$ and $k \ge 2$, if $G$ is a simple graph on $n$ vertices with $n+k$ edges, then $G$ contains a cycle of length at most $$\frac{14(n+k)\log{k}}{3k}.$$
\end{corollary}
\begin{proof}
By Theorem \ref{bollobas_thm}, we have that the girth is at most:
$$\frac{2(n+k)}{3k}(\log k+\log \log k + 4) \le \frac{14(n+k)\log{k}}{3k}$$
since that is equivalent to:
$$\log{\log{k}}+4 \le 6\log{k}.$$
To see that this is true, let $f(k) = 6\log{k} - \log{\log{k}}-4$. Then $f(2) = 6-4=2 \ge 0$, and for all $k \ge 2$ we have:
\begin{align*}
    f'(k) = \frac{6}{k \ln{2}} - \frac{1}{\log{k}\ln{2}}\frac{1}{k \ln{2}} \ge \frac{4}{k \ln{2}} \ge 0.
\end{align*}
It follows that $f(k) \ge 0$ for all $k \ge 2$, as desired. This proves Corollary \ref{bollobas_cor}.\end{proof}

Two final results we will make use of is a set of Chernoff bounds and Chebyshev's Inequality:
\begin{theorem}[\cite{chernoff}]\label{chernoff_thm}
Let $\{X_i\}_{i=1}^m$ be independent indicator random variables, and let $X = \sum_{i=1}^m X_i$. Then for any $\epsilon > 0$, we have:
\begin{align*}
    \mathbb{P}(X \le (1-\epsilon)\mathbb{E}[X]) &\le \exp\left(-\frac{\epsilon^2}{2}\mathbb{E}[X]\right);
    \\
    \mathbb{P}(X \ge (1+\epsilon)\mathbb{E}[X]) &\le \exp\left(-\frac{\epsilon^2}{2+\epsilon}\mathbb{E}[X]\right).
\end{align*}
\end{theorem}

\begin{theorem}[Chebyshev's Inequality]\label{chebyshev}
Let $X$ be a random variable with finite expected value $\mu$ and finite non-zero variance $\sigma^2$. Then for any real number $k > 0$ we have:
\begin{align*}
    \mathbb{P}(|X-\mu| \ge k \sigma) \le \frac{1}{k^2}.
\end{align*}
\end{theorem}

\section{$n+ck$ colors}
We first consider a relaxation of Conjecture \ref{aharoni_thm} where we have $n+c_1k$ color classes each of size at least $c_2k$, for constants $c_1,c_2$ which we will specify. In this case, we obtain upper bounds for the rainbow girth that are stronger than $\lceil n/k \rceil$ to a surprising degree. For this reason, these results are interesting in their own right. They are also used in the proof of our main result in the next section.

Our first result is the following:
\begin{theorem}\label{res_one}
Let $k > 1$ be an integer, and let $G$ be a simple graph on $n$ vertices. Suppose we have a coloring of the edges of $G$ with $n+k$ color classes of size at least $ck$, where $c=10^9$. Then $G$ has rainbow girth at most $6$ or $G$ has rainbow girth at most:
$$\frac{n(\log k)^2}{10k^{3/2}}+14\log k.$$
\end{theorem}
\begin{proof}
Since the graph is simple, we have that $n^2 \ge |E(G)| \ge cnk$ and thus we may assume that $n \ge ck$. Now, we claim there exists a set of vertices $S$ with $|S| \le n \log{k} / (140\sqrt{k})$ such that every color class has at least one edge incident to a vertex in $S$. To see this, we let $s = \lfloor 2\log{k} \rfloor$ and $t = \lfloor n/(560\sqrt{k}) \rfloor$. We will iteratively construct $s$ sets of vertices $S_1, \cdots, S_s$, each of size at most $t$, as follows. Suppose we have constructed $S_1, \cdots, S_i$ so far. Let $T_i = \cup_{j=1}^i S_j$, and let $C_i$ denote the set of colors whose color class has no edge incident to a vertex in $T_i$. Let $H$ be a random set of $t$ vertices chosen uniformly with repetition. For any color class $a$, note that the number of vertices which are incident to an edge of color $a$ is at least $\sqrt{ck}$, since if there are at most $\sqrt{ck}$ vertices incident to edges of color $a$, the number of edges of color $a$ will be at most $ck/2$. Also, we have that $t \ge n/(560\sqrt{k})-1 \ge n/(1120\sqrt{k})$ since $n \ge 1120\sqrt{k}$ which is implied  by $n \ge ck$. Using these two observations, we have that the expected number of colors in $C_i$ whose color class has no edges incident to the vertices of $H$ is at most:
$$
    \left(1-\frac{\sqrt{ck}}{n}\right)^{t} |C_i| \le
    \left(1-\frac{\sqrt{ck}}{n}\right)^{n/(1120\sqrt{k})} |C_i| \le 
    e^{-\sqrt{c}/1120} |C_i|. \\
$$

Thus, we can choose $S_{i+1}$ such that $|C_{i+1}| \le e^{-\sqrt{c}/1120} |C_i|$, and iterate. When we finish, we have a collection of sets $\{S_1,S_2,\cdots,S_s\}$ such that:
\begin{align*}
    |C_s| \le e^{-\sqrt{c}s/1120}(n+k) &\le 
    e^{-\sqrt{c}(2\log{k}-1)/1120}(n+k)  \\ &\le 2ne^{-\sqrt{c}(\log{k})/1120} \\
    &\le \frac{n\log{k}}{280\sqrt{k}}
\end{align*}

where that last inequality is true for $k \ge 2$ since:
\begin{align*}
    2nk^{-\sqrt{c}/(1120\ln{2})} \le \frac{n\log{k}}{280\sqrt{k}} \iff 560 \le k^{\sqrt{c}/(1120\ln{2})-1/2}\log{k}
\end{align*}
which is true for $k=2$ and thus for all $k \ge 2$. Now, we have that $T_s$ is a set of vertices with $|T_s| \le \frac{n\log{k}}{280\sqrt{k}}$ such that at most $\frac{n\log{k}}{280\sqrt{k}}$ colors $a$ have no edge of their color class adjacent to any of the vertices in $T_s$. It follows that, by adding at most $\frac{n\log{k}}{280\sqrt{k}}$ vertices, we can find a set of vertices of size at most $\frac{n\log{k}}{140\sqrt{k}}$ which is incident to at least one edge of every color class, as desired.

Now, let $S$ be a set of at most $\frac{n\log{k}}{140\sqrt{k}}$ vertices such that $S$ is incident to at least one edge of every color. For each color $a$, choose one edge $e_c$ of color $a$ such that $e_c$ is incident to at least one vertex in $S$. Let $E$ be the set of these chosen edges $e_c$. Then $|E| = n+k$ and $E$ contains exactly one edge of each color. Now, let $H$ be the subgraph with $V(H) = \bigcup_{(uv) \in E}\{u,v\}$ and $E(H) = E$, and let $S = \{v_1,v_2,\cdots,v_p\}$, where $p = |S|$. Partition $V(H) \setminus S$ into $X_1,\cdots,X_p$ such that $X_i \subseteq N_H(v_i)$ for all $1 \le i \le p$. Now, contract each $H_i = X_i \cup \{v_i\}$ to a single vertex (by contracting each edge of $H_i$ iteratively), and let the resulting graph be $H'$. We have that $|V(H')| = |S| \le \frac{n\log{k}}{140\sqrt{k}}$ and $|E(H')| = |S|+k$. Note that a rainbow cycle $C$ in $H'$ corresponds to a rainbow cycle in $G$ with length at most $3|C|$, by replacing each contracted vertex by at most a two-edge path. We may assume that $H'$ is simple, since otherwise we obtain a rainbow cycle of length at most $6$ in $G$. Then applying Corollary $\ref{bollobas_cor}$ to $H'$ gives a rainbow cycle in $G$ of length at most:
\begin{align*}
    \frac{14\left(\frac{n\log{k}}{140\sqrt{k}}+k\right)\log{k}}{k} = \frac{n(\log k)^2}{10k^{3/2}}+14\log k
\end{align*}
as desired. This proves Theorem \ref{res_one}.\end{proof}

We immediately obtain the following interesting corollary:

\begin{corollary}\label{cor_one}
Let $k > 1$ be an integer, and let $G$ be a simple graph on $n$ vertices. Suppose that we have a coloring of the edges of $G$ with $n+k$ color classes of size at least $ck$, where $c=10^9$, and suppose also that $\frac{140k^{3/2}}{\log{k}} \le n$. Then $G$ has rainbow girth at most $\frac{n(\log k)^2}{5k^{3/2}}$.
\end{corollary}
\begin{proof}
The condition on the size of $n$ is equivalent to:
$$14\log{k} \le \frac{n(\log{k})^2}{10k^{3/2}}.$$
We have that $G$ has rainbow girth at least $7$ since $\frac{n(\log k)^2}{5k^{3/2}} \ge 7$ is implied by the condition. Then Corollary \ref{cor_one} gives that $G$ has rainbow girth at most:
\begin{align*}
    \frac{n(\log k)^2}{10k^{3/2}}+14\log k \le \frac{n(\log k)^2}{5k^{3/2}}
\end{align*}
as desired. This proves Corollary \ref{cor_one}.\end{proof}

If we would like rainbow girth to be at most roughly $n/k^{3/2}$ (as promised by Corollary \ref{cor_one}), it is necessary that $k^{3/2} < n$, since a simple graph cannot have rainbow girth less than three. Corollary \ref{cor_one} can be interpreted as saying that, for the region where it makes sense (where $k^{3/2} < n$, roughly), when we relax the number of colors slightly from $n$ to $n+k$, we obtain a much shorter rainbow cycle of length at most approximately $n/k^{3/2}$, in comparison to the tight bound of $n/k$ for the case of $n$ colors.

Next, we present a result of a similar flavor for the case where $k$ is large relative to $n$.
\begin{theorem}\label{res_two}
Let $k > 1$ be an integer, and let $G$ be a simple graph on $n$ vertices. Suppose we have a coloring of the edges of $G$ with $n+k$ color classes of size at least $ck$, where $c=10^9$, and also that $140k^{10/9} \ge n$. Then $G$ has rainbow girth at most $6$.
\end{theorem}
\begin{proof}
We may assume that $n \ge ck$ since otherwise we have at least $nck > n^2$ edges which is a contradiction since $G$ is simple. Let a \emph{colorful star} be a subgraph $H$ of $G$ such that $H$ is a star with at least $\frac{ck^2}{4n}$ edges such that no color appears more than $c^{2/3} k^{2/3}$ times in $E(H)$. Let a \emph{collection of colorful stars} be a set $C =\{H_1,H_2,\cdots,H_m\}$ of colorful stars such that every color appears in at most one of the $E(H_i)$. For a collection $C$ of colorful stars, for $1 \le i \le m$, let $v_i$ be the center of the star $H_i$, and let $V(C) = \{v_1,\cdots,v_p\}$ and $E(C) = \cup_{i=1}^p E(H_i)$ be the set of all star centers and the set of all edges, respectively.

Now, let $C$ be a collection of colorful stars in $G$, chosen to be maximal with respect to the number of stars. We first prove the following claim, which says that the number of colors appearing in $E(C)$ is large.

\begin{claim}\label{claim1}
At most $k/2$ colors do not appear in $E(C)$.
\end{claim}
\begin{proof}
Suppose not. Let $S$ be the set of colors which do not appear in any of the $E(v_i)$. Note that $|S| > k/2$. For each color $s \in S$, for $v \in V(G)$ let $d_s(v)$ be the number of edges incident to $v$ of color $s$, and set $d'_s(v) = d_s(v)$ if $d_s(v) \le c^{2/3}k^{2/3}$, and otherwise set $d'_s(v) = 0$. Now, let $H$ be the set of vertices $v \in V(G)$ with $d_s(v) > c^{2/3} k^{2/3}$. Note that $|H| < 2ck/(c^{2/3}k^{2/3}) = 2c^{1/3}k^{1/3}$. Then the number of edges of color $s$ with both ends in $H$ is at most $4c^{2/3}k^{2/3}$, so it follows that:
$$\sum_{v \in V(G)} d'_s(v) \ge ck - 4c^{2/3} k^{2/3} \ge \frac{ck}{2}$$
since the last inequality is equivalent to $k \ge 8^3 / c$ which is true for $k > 1$ since $c = 10^9$. Then, on average, a vertex $v$ has:
\begin{align*}
    \sum\limits_{s \in S} d_s'(v) > \frac{ck^2}{4n}.
\end{align*}
Now, let $v$ be a vertex for which $\sum\limits_{s \in S} d_s'(v) > \frac{ck^2}{4n}$, and construct a colorful star with center $v$ and $d_s'(v)$ edges of color s incident with $v$ for all $s \in S$. Then we can add $v$ to $C$ and obtain a larger collection of colorful stars, which contradicts the maximality of $C$. It follows that there are at most $k/2$ colors which do not appear in $E(C)$, as desired. This proves Claim \ref{claim1}.\end{proof}

We now prove a second claim, which says the number of colorful stars in $C$ is small.

\begin{claim}\label{claim2}
$|C| < \frac{n^{1/5}}{12}$.
\end{claim}
\begin{proof}
Suppose not; then $|C| \ge \frac{n^{1/5}}{12}$. It suffices to show a contradiction for the case where $t = |C| = \lceil\frac{n^{1/5}}{12}\rceil$, so that $\frac{n^{1/5}}{12} \le t < \frac{n^{1/5}}{12} + 1 \le \frac{n^{1/5}}{6}$ since $n \ge ck \ge 10^9k \ge 12^5$. Now, for a colorful star $H_i$ with center $v_i$, let $M(H_i) = V(H_i) \setminus \{v_i\}$. We claim that for any two colorful stars $H_i,H_j \in C$ with centers $v_i$ and $v_j$, if $H_{ij} = M(H_i) \cap M(H_j)$, then either $v_1$ has all its edges in $H_1$ to $H_{ij}$ in the same color class, or $v_2$ has all its edges in $H_2$ to $H_{ij}$ in the same color class. Suppose not. Then without loss of generality there are two edges $e_1=(v_i,w_1)$ and $e_2=(v_i,w_2)$ for $w_1,w_2 \in H_{ij}$ such that $e_1$ and $e_2$ have colors $a_1$ and $a_2$ with $a_1 \ne a_2$. Let $a_3$ be the color of $(v_j,w_1)$. Then clearly $a_3$ is also the color of $(v_j, w_2)$, since otherwise we obtain a rainbow cycle of length $4$. Now, consider an arbitrary edge $(v_j,w_3)$ to a vertex $w_3 \in H_{ij}$ with $w_3 \notin \{w_1,w_2\}$. We claim that $(v_j,w_3)$ must have color $a_3$. Indeed, if $(v_i,w_3)$ does not have color $a_1$ then the $4$-cycle $(v_i,w_1,v_j,w_3)$ implies that $(v_j,w_3)$ has color $a_3$, and if $(v_i,w_3)$ does not have color $a_2$ then the $4$-cycle $(v_i,w_2,v_j,w_3)$ implies that $(v_j,w_3)$ has color $a_3$. Since $(v_i,w_3)$ cannot have both color $a_1$ and color $a_2$ it follows that $(v_j,w_3)$ has color $a_3$ for all $w_3 \in H_{ij}$, as desired.

This implies that for all $v_i,v_j \in V(C)$ we have $|M(v_i) \cap M(v_j)|\le c^{2/3}k^{2/3}$. Then it follows that every colorful star $H_i$ has at least:
$$\frac{ck^2}{4n} - t c^{2/3} k^{2/3}$$
vertices in $M(H_i)$ which are not in $\cup_{j\ne i} M(H_j)$. The condition $140k^{10/9} \ge n$ implies $k \ge \frac{n^{9/10}}{140^{9/10}}$, and we have that:
\begin{align*}
    \frac{tck^2}{8n} \ge tc \frac{n^{9/5}}{8n140^{9/5}} \ge \frac{10^9 n}{96 \cdot 140^{9/5}} > n.
\end{align*}
We also have that:
$$\frac{tck^2}{8n} > t^2 c^{2/3} k^{2/3}$$
since the inequality is equivalent to $c^{1/3} k^{4/3} > 8nt$, which is true since (using $k^{10/9} \ge n/140$ and $t \le n^{1/5}/6$ from above):
\begin{align*}
    c^{1/3}k^{4/3} > \frac{8 \cdot 140^{6/5} k^{4/3}}{6} \ge \frac{8n^{6/5}}{6} \ge 8nt.
\end{align*}
Then we have that:
\begin{align*}
    \left|\bigcup_{H_i \in C} V(H_i) \right| &\ge t\left(\frac{ck^2}{4n} - t c^{2/3} k^{2/3}\right)\\ &= \frac{tck^2}{8n} + \frac{tck^2}{8n} - t^2 c^{2/3} k^{2/3} \\ 
    &> n + t^2 c^{2/3}k^{2/3}-t^2 c^{2/3}k^{2/3} \\ &> n
\end{align*}

which gives a contradiction. This proves Claim \ref{claim2}.\end{proof}

Now, for each color class with at least one edge in $E(C)$, we choose exactly one such edge. Let the resulting set of edges be $F$; from Claim $\ref{claim1}$, we know that $|F| \ge n + \frac{k}{2}$. Now, let $H$ be the subgraph with $V(H) = \bigcup_{(uv) \in F}\{u,v\}$ and $E(H) = F$, and let $S = \{v_1,v_2,\cdots,v_p\}$, where $p = |S|$. Partition $V(H) \setminus S$ into $X_1,\cdots,X_p$ such that $X_i \subseteq N_H(v_i)$ for all $1 \le i \le p$. Now, contract each $H_i = X_i \cup \{v_i\}$ to a single vertex, and let the resulting graph be $H'$. By Claim \ref{claim2}, we have that $|V(H')| < \frac{n^{1/5}}{12}$, and, since $k^{10/9} \ge n/140$ and $n \ge c=10^9$, we obtain: $$|E(H')| = |V(H')| + \frac{k}{2} \ge \frac{k}{2} \ge \frac{n^{9/10}}{2 \cdot 140^{9/10}} > \frac{n^{2/5}}{144} > |V(H')|^2.$$
Thus we obtain a rainbow cycle of length at most $2$ in $H'$, which gives a rainbow cycle of length at most $6$ in $H$, as desired. This proves Theorem \ref{res_two}.\end{proof}

We conclude this section with an immediate corollary of the above results which will be used in the proof of the next section:
\begin{corollary}\label{cor_to_use}
Let $k > 1$ be an integer, and let $G$ be a simple graph on $n$ vertices. Suppose we have a coloring of the edges of $G$ with $n+k$ color classes of size at least $ck$, where $c=10^9$. Then $G$ has rainbow girth at most $n/k$.
\end{corollary}
\begin{proof}
If $28k\log{k} \le n$, then by Theorem \ref{res_one} we have rainbow girth at most:
$$\frac{n(\log k)^2}{10k^{3/2}}+14\log k \le \frac{n}{2k} + \frac{n}{2k} = \frac{n}{k}$$
since $(\log{k})^2 \le 5\sqrt{k}$ holds for $k \ge 2$. To see this, note that it is equivalent to $\log{k} \le \sqrt{5}k^{1/4}$. Let $f(k) = \sqrt{5}k^{1/4}-\log{k}$. We compute:
\begin{align*}
    f'(k) = \frac{\sqrt{5}}{4k^{3/4}} - \frac{1}{k \ln{2}}
\end{align*}
and it follows that $f(k)$ achieves its minimum for $k_0 \ge 2$, $k \in \mathbb{R}$ at the point $k_0 = \left(\frac{4}{\sqrt{5}\ln{2}}\right)^4$. We verify that $f(k_0) \ge 0$, so it follows that $f(k) \ge 0$ for all $k \ge 2$, as desired.

If $28k\log{k} > n$, we claim that $140k^{10/9} \ge n$. Indeed, $140k^{10/9} > 28k\log{k}$ is equivalent to $5k^{1/9}>\log{k}$ which is true for $k \ge 2$. To see this, by taking derivatives as before it suffices to verify that the inequality is true for $k_0$ such that $k_0^{1/9} = 9/(5\ln{2})$, which is true. Then, Theorem $\ref{res_two}$ gives that $G$ has rainbow girth at most $6$. Since $n^2 \ge |E(G)| \ge cnk$, we have that $n/k \ge c \ge 6$, so it follows that $G$ has rainbow girth at most $n/k$, as desired. This proves Corollary \ref{cor_to_use}.\end{proof}

\section{$n$ colors}
Now we are ready to prove Theorem \ref{main_result}, which we restate:
\begin{theorem}\label{main}
Let $k \ge 1$ be an integer, and let $G$ be a simple graph on $n$ vertices. Suppose we have a coloring of the edges of $G$ with $n$ color classes of size at least $ck$, where $c = 10^{11}$. Then $G$ has rainbow girth at most $n/k$.
\end{theorem}
\begin{proof}
If $k=1$, then taking one edge of each color gives a rainbow cycle of length at most $n$. So we may assume $k > 1$. Also, since $G$ is simple, we have that the number of edges $|E(G)|$ satisfies $n^2 \ge |E(G)| \ge nck$, and thus we may assume that $n \ge ck$. Now, let $t=ck$. By removing edges if necessary, we may assume that every color class has exactly $t$ edges. Now, we say that a color $a$ \emph{dominates} a vertex $v \in V(G)$ if there are at least $\frac{t}{100}+8k$ edges incident to $v$ with color $a$. Call a vertex $v$ \emph{color-dominated} if there exists a color $a$ which dominates $v$, and call a color $a$ \emph{vertex-dominating} if there exists a vertex $v$ which is dominated by $a$. The definition is motivated by a desire to reduce to the case of the Caccetta-H\"aggkvist conjecture as in \cite{aharoni}, where each color class is a star centered at a different vertex. A color being vertex-dominating means that its edges form a large star, which will be useful in applying existing approximate results for the Caccetta-H\"aggkvist conjecture. Now, for each vertex-dominating color $a$, pick one vertex $v_a$ dominated by $a$ (not necessarily unique), and let the resulting set of vertices be $S$. Let $H = V(G) \setminus S$.

Suppose first that $|H| \le \frac{t}{100}$. Let $b$ be the coloring of the edges. We construct a digraph $G'$ with $V(G') = S$, and for all $i, j$ with $v_i, v_j \in S$, there is an arc $v_i \rightarrow v_j$ if $v_iv_j \in E(G)$ and $b(v_iv_j) = i$. Every vertex $v_i$ is incident with at least $\frac{t}{100}+8k$ edges $e$ with $b(e) = i$, and since $|H| \le \frac{t}{100}$, there are at least $8k$ edges $e = v_iu$ with $b(e) = i$ and $u \in S$. Therefore, $\delta^+(G') \geq 8k$. 

Now, we claim $n/(8k) + 74 \le n/k$, which is equivalent to $n \ge \frac{592k}{7}$ which is true since $n \ge ck = 10^{11}k$.

Then, by applying Theorem \ref{shen_thm} to $G'$ we obtain a directed cycle $K$ of length at most $\lceil n/(8k) \rceil + 73 \le n/k$ in $G'$. The edges of $G$ that correspond to arcs of $K$ form a rainbow cycle of length at most $n/k$ in $G$.

So we may assume that $|H| > \frac{t}{100}$. Let $r = |H|$, so we have $\frac{t}{100} < r \le n$. Let $T \subseteq H$ be a random set of vertices in $H$ where each vertex in $H$ is included in $T$ independently with probability $\frac{4k}{r}$.

Now consider a color $a$ which does not dominate a vertex in $S$ (and thus does not dominate any vertex). We will show that the probability that $a$ has at least $t/100$ edges with both ends in $G \setminus T$ is at least $1-\frac{k}{2r}$. We claim that we may assume all of the edges of $a$ have both ends in $H$. Indeed, if this is not the case, perform the following iterative process while there is still an edge $e$ of color $a$ not contained in $H$. 

If $e$ has both ends in $G \setminus H$, then remove $e$. Now, note that at most $200$ vertices are incident to at least $t/100$ edges of color $a$. Since $|H| > \frac{t}{100} = 2^9k$, there exists a pair of vertices $v_1,v_2 \in H$ such that there is no edge of color $a$ between $v_1$ and $v_2$ and $v_1,v_2$ are both incident to less than $t/100$ edges of color $a$. Then add an edge of color $a$ between $v_1$ and $v_2$. If instead $e$ has one end in $H$, say the vertex $w$, then remove $e$ and add an edge of color $a$ from $w$ to any vertex $v \in H$ such that there is not already an edge between $v$ and $w$ of color $a$ and both $v$ and $w$ are incident to less than $t/100$ edges of color $a$. Repeat this process until we obtain a graph $G'$ where all the edges of $a$ have both ends in $H$. If we can show that for $G'$ the probability that $a$ has at least $t/100$ edges in $G' \setminus T$ is at least $1-\frac{k}{2r}$, then it clearly follows that the probability that $a$ has at least $t/100$ edges in $G \setminus T$ is also at least $1-\frac{k}{2r}$. Thus, we may assume without loss of generality that all of the edges of $a$ are contained in $H$, as claimed.

Now, let the edges of $a$ be $e_1,\cdots,e_t$. Let the random variable $E_i$ have value $1$ if $e_i \in G \setminus T$ and have value $0$ otherwise. Let $E = \sum_{i=1}^t E_i$, and for a random variable $R$ let $\text{Var}(R)$ denote the variance of $R$. Since $a$ is not vertex-dominating, we have that each edge $e_i$ shares an end with at most $\frac{t}{50}+16k$ edges of the same color. It follows that each $E_i$ is dependent on at most $\frac{t}{50}+16k$ of the variables $\{E_1,E_2,\cdots,E_t\}$. Let $x = \frac{r-4k}{r}$, and note that the probability that an edge $e_i$ is in $G \setminus T$ is simply $x^2$, so for all $1 \le i \le t$ we have that $E_i$ is a Bernoulli random variable with probability equal to $x^2$. Then it follows that $\text{Var}(E_i) = x^2(1-x^2)$ and furthermore, if $e_i$ and $e_j$ share an end, we obtain:
\begin{align*}
    \text{Cov}(E_i,E_j)&=\mathbb{E}(E_i E_j) - \mathbb{E}(E_i)\mathbb{E}(E_j) \\
    &= x^3 - x^4
\end{align*}
and thus we have:
\begin{align*}
    \text{Var}(E) &= \sum_{i=1}^t \text{Var}(E_i) + \sum_{1 \le i \ne j \le t} \text{Cov}(E_i,E_j) \\
    &\le f(x) := tx^2(1-x^2)+\left(\frac{t}{50}+16k\right)t(x^3-x^4).
\end{align*}

\begin{claim}\label{var}
Let $\alpha = 1-\frac{400}{c}$. For all $\alpha \le y < 1$ we have:
$$f(y) \le t^2 \left(y^2 - \frac{1}{100}\right)^2 \frac{k}{2r}.$$
\end{claim}

\begin{proof}
Since $t/100 < r$, we have that $\frac{100}{c} > \frac{k}{r} > 0$ and thus $\alpha = 1-\frac{400}{c} < y < 1$. Define $g(y)$ as follows:
\begin{align*}
    g(y) = t^2\left(y^2-\frac{1}{100}\right)^2 \frac{1-y}{8}.
\end{align*}
We claim that $f(y) \le g(y)$ for all $\alpha \le y < 1$. To see this, let $h(y) = g(y)-f(y)$. Then $h(y) \ge 0$ is equivalent to:
\begin{align*}
    h_1(y) = \frac{h(y)}{t(1-y)} =  \frac{t\left(y^2-\frac{1}{100}\right)^2}{8}-y^2-\left(\frac{t}{50}+16k+1\right)y^3 \ge 0.
\end{align*}
We claim that $h_1(\alpha) \ge 0$ and $h_1'(y) \ge 0$ for all $\alpha \le y < 1$. For the first claim, since $t = ck = 10^{11}k$ and $\alpha = 1-\frac{400}{c}$, we have that:
\begin{align*}
    \frac{t\left(\alpha^2-\frac{1}{100}\right)^2}{8} \ge \frac{t}{32} 
    \ge \frac{t}{50}+16k+2
    \ge \alpha^2 + \left(\frac{t}{50}+16k+1\right)\alpha^3.
\end{align*}
To show $h_1'(y) \ge 0$ for all $\alpha \le y < 1$, we compute:
\begin{align*}
    h_1'(y) = \frac{t}{2}\left(y^2-\frac{1}{100}\right)y-2y-3\left(\frac{t}{50}+16k+1\right)y^2.
\end{align*}
Since $y>0$, $h_1'(y) \ge 0$ is equivalent to $h_2(y) \ge 0$, where:
\begin{align*}
    h_2(y) = \frac{t}{2}\left(y^2-\frac{1}{100}\right)-2-3\left(\frac{t}{50}+16k+1\right)y.
\end{align*}
Now, we claim that $h_2(\alpha) \ge 0$ and $h_2'(y) \ge 0$ for $\alpha \le y < 1$. The first claim follows from the facts $t=ck=10^{11}k$ and $\alpha = 1-\frac{400}{c}$:
\begin{align*}
    \frac{t\left(\alpha^2-\frac{1}{100}\right)}{2} \ge \frac{t}{4} \ge 2+3\left(\frac{t}{50}+16k+1\right) \ge 2+3\left(\frac{t}{50}+16k+1\right)\alpha.
\end{align*}
To show $h_2'(y) \ge 0$ for all $\alpha \le y < 1$, we compute:
\begin{align*}
    h_2'(y) = ty-3\left(\frac{t}{50}+16k+1\right).
\end{align*}
Now, $t\alpha-3\left(\frac{t}{50}+16k+1\right) \ge 0$ for all $k \ge 1$, so it follows that $h_2'(y) \ge h_2'(\alpha) \ge 0$ for all $\alpha \le y < 1$. This implies that $h_1(y) \ge 0$ for all $\alpha \le y < 1$, which in turn gives $h(y) \ge 0$ for all $\alpha \le y < 1$. Thus $f(y) \le g(y)$ for all $\alpha \le y < 1$, and we obtain:
\begin{align*}
    f(y) \le g(y) \le t^2 \left(y^2 - \frac{1}{100}\right)^2 \frac{k}{2r}.
\end{align*}
as desired. This completes the proof of Claim \ref{var}.\end{proof}

Now, Claim \ref{var} gives:
\begin{align*}
    \text{Var}(E) \le f(x) \le  t^2 \left(x^2 - \frac{1}{100}\right)^2 \frac{k}{2r}.
\end{align*}
Let $\lambda = t\left(x^2 - \frac{1}{100}\right)$. Then we have shown that $\text{Var}(E) \le \lambda^2 \frac{k}{2r}$. Let $q$ be the probability that the color $a$ has at least $t/100$ of its edges in $G \setminus T$. Note that $\mathbb{E}(E) = tx^2$, so:
\begin{align*}
    1-q &= \mathbb{P}\left(E \le t/100 \right) = \mathbb{P}(\mathbb{E}(E)-E \ge \lambda).
\end{align*}
Then by Theorem \ref{chebyshev} (Chebyshev's Inequality), we have:
\begin{align*}
    1-q \le \mathbb{P}(|E-\mathbb{E}(E)| \ge \lambda) \le \frac{\text{Var}(E)}{\lambda^2} \le \frac{k}{2r}.
\end{align*}
We say that a color $a$ is \emph{bad} if $a$ is not vertex-dominating and $a$ has less than $t/100$ of its edges in $G \setminus T$. Let $B$ be the set of bad colors, and let $Y = |B|$. Since $1-q \le \frac{k}{2r}$, we have that $\mathbb{E}(Y) \le k/2$. It follows from Markov's Inequality that $\mathbb{P}(Y \ge k) \le 1/2$. Recall that $T$ was formed by choosing each vertex in $H$ independently with probability $4k/r$. Then $\mathbb{E}(|T|) = 4k$. Applying Theorem \ref{chernoff_thm} yields that for all $k \ge 2$:
$$\mathbb{P}(|T| \ge 8k) + \mathbb{P}(|T| \le 2k) \le \exp(-4k/3) + \exp(-k/2) < 1/2.$$
Since $k > 1$ is an integer, it follows that with positive probability we have both $2k < |T| < 8k$ and $Y < k$, so there exists a set $T \subset G \setminus S$ with $|T| \le 8k$ and such that $|T|-Y \ge k$. If $G' = G \setminus T$, then since $|T| \le 8k$ it follows that for every vertex-dominating color class at least $t/100$ of its edges are in $G'$. Then we have that at least $|V(G')|+k$ colors $a$ have at least $t/100$ edges in $G'$. Applying Corollary \ref{cor_to_use} to $G'$ gives that $G'$ has rainbow girth at most $|V(G')|/k$ and thus $G$ has rainbow girth at most $n/k$, as desired. This completes the proof.\end{proof}
\section{Further Work}
There are a number of directions further research in this area can go. Here we mention a few of our favorites. One direction is to prove tight results, such as proving that Conjecture \ref{aharoni_thm} is true for $k = 3$. Another direction is to try to improve the constant from our proof; we made little effort to optimize it. 

Another interesting question is whether there exist extremal examples for Conjecture \ref{aharoni_thm} which are not inherited from Conjecture \ref{ch_thm}, namely that do not have the property that for each vertex $v \in V(G)$ there exists a color $a$ whose color class is the edge set of a star centered at $v$. Finally, a related problem which we did not consider is a relaxation of Conjecture \ref{aharoni_thm}, which is the following:

\begin{conjecture}[\cite{devos}]\label{proper}
Let $n,k$ be positive integers, and let $G$ be a simple graph on $n$ vertices. Let $b$ be a coloring of the edges of $G$ with $n$ color classes of size at least $k$; then $G$ has a cycle $C$ of length at most $\lceil n/k \rceil$ such that no two incident edges of $C$ are the same color.
\end{conjecture}

This conjecture is interesting because it still implies Conjecture \ref{ch_thm}, but seems like it might be substantially easier than Conjecture \ref{aharoni_thm}, as it deals with a local condition rather than a global condition. However, we suspect it may require different methods than those used in this paper.

Finally, nowhere in this paper did we use induction, while a number of the results for Conjecture \ref{ch_thm} utilize induction. Is there a way to use inductive arguments in this context?


\pagebreak
\section{Corrigendum - added September 9, 2022}

An error in the paper was pointed out to the authors via private correspondence by Yuhui Cheng, for which we are very grateful. The error is in the part of the proof of Theorem \ref{main} which concerns the iterative process which allows us to assume without loss of generality that all the edges of $a$ are contained in $H$. In particular, the issue is in the line ``\emph{If instead e has one end in $H$, say the vertex w, then remove $e$ and add an edge of
color $a$ from $w$ to any vertex $v \in  H$ such that there is not already an edge between $v$ and
$w$ of color $a$ and both $v$ and $w$ are incident to less than $t/100$ edges of color $a$.}'' This is not always possible; in fact it is possible that there is an edge of color $a$ from $w$ to every other vertex in $H$ which is incident to less than $t/100$ edges of color $a$. There are at least $|V(H)| - 200 > t/100 - 200$ vertices in $H$ incident to fewer than $t/100$ edges of colour $a$; so when this process fails, $w$ is already adjacent with an edge of colour $a$ to at least $t/100 - 200$ vertices in $H$. 

We resolve this issue as follows. If we encounter this situation, we delete the rest of the edges of color $a$ incident with $w$ and a vertex in $G \setminus H$, and then continue the iterative process. We claim that this will delete at most $0.01t$ edges of color $a$. Indeed, note that for each vertex $w \in H$ that we delete such edges for, we delete at most $8k+200 \le 208k$ edges, and the number of such problematic vertices is at most:
\begin{align*}
    \frac{2t}{\frac{t}{100}-200} \le \frac{2 \cdot 10^{11}}{10^9-200}
\end{align*}
Then we have that the total number of edges of color $a$ which are deleted is at most:
\begin{align*}
    \frac{416 \cdot 10^{11}}{10^9-200} k < 0.01 \cdot 10^{11} k
\end{align*}
since $416 < 0.01(10^9 - 200)$. Therefore we preserve at least $0.99t$ edges from each color class. We finish by first noting that the arguments showing Corollary \ref{cor_to_use} go through with $c = 0.99 \cdot 10^9$, and also that the rest of the proof of Theorem \ref{main} goes through with $0.99 \cdot 10^{11} k$ edges instead of $ 10^{11} k$ edges. It follows that the result of Theorem \ref{main} still holds.

\end{document}